\documentclass[leqno]{amsart}
\usepackage{amsfonts,amssymb,amsmath,amsgen,amsthm}
\usepackage{hyperref}

\theoremstyle{plain}
\newtheorem{theorem}{Theorem}[section]
\newtheorem{definition}[theorem]{Definition}

\newtheorem{lemma}[theorem]{Lemma}
\newtheorem{corollary}[theorem]{Corollary}
\newtheorem{proposition}[theorem]{Proposition}
\theoremstyle{remark}
\newtheorem{remark}[theorem]{Remark}

\def\C{{\mathbf C}}% complex numbers
\def\R{{\mathbf R}}% real numbers
\def\F{\mathcal F}% Fourier transform

\def\virgp{\raise 2pt\hbox{,}}

\def\({\left(}
\def\){\right)}
\def\<{\left\langle}
\def\>{\right\rangle}
\def\le{\leqslant}
\def\ge{\geqslant}

\def\1{{\bf 1}}

\def\Tend#1#2{\mathop{\longrightarrow}\limits_{#1\rightarrow#2}}

\def\d{{\partial}}
\def\l{\lambda}
\def\g{\gamma}

\def\Om{\Omega}
\def\si{{\sigma}}

\DeclareMathOperator{\RE}{Re}
\DeclareMathOperator{\IM}{Im}

\numberwithin{equation}{section}

\begin{document}

\title[Extinction for the Schr\"odinger equation]{Finite time
  extinction by nonlinear damping for the Schr\"odinger equation}  

\author[R. Carles]{R\'emi Carles}
\author[C. Gallo]{Cl\'ement Gallo}
\address{Univ. Montpellier~2\\Math\'ematiques \\
  CC~051\\F-34095 Montpellier} 
\address{CNRS, UMR 5149\\  F-34095 Montpellier\\ France}
\email{Remi.Carles@math.cnrs.fr}
\email{cgallo@univ-montp2.fr}

\thanks{This work was supported by the French ANR project
  R.A.S. (ANR-08-JCJC-0124-01)}   
  
\begin{abstract} 
We consider the Schr\"odinger equation on a compact manifold, in the
presence of a nonlinear damping term, which is homogeneous and
sublinear. For initial data in the energy space, we construct a 
weak solution, defined for all positive time, which is shown to be
unique. In the one-dimensional 
case, we show that it becomes zero in finite time. In
the two and three-dimensional cases, we prove the same result under the
assumption of extra regularity on the initial datum.
\end{abstract}
\maketitle

\section{Introduction}
\label{sec:intro}

We consider the Schr\"odinger equation with a homogeneous damping term,
\begin{equation}\label{nls}
i\frac{\partial u}{\partial t}+\Delta u= -i\g\frac{u}{|u|^\alpha},\quad
t\in \R_+,\  x\in M\quad ;\quad u_{\mid t=0}=u_0,
\end{equation}
where $\g>0$, $0<\alpha\le 1$,  $(M,g)$ is a smooth compact Riemannian
manifold of dimension $d$, and $u$ is complex-valued. If $d\le 3$, we
prove that if the initial 
datum $u_0$ is sufficiently regular (in $H^1(M)$ if $d=1$, in $H^2(M)$
if $d=2,3$), then every weak solution to \eqref{nls} becomes zero in
finite time. 
  The reason why the space variable belongs to a compact manifold
  and not to the whole Euclidean space is most likely purely
  technical. It seems sensible to believe that the
  extinction phenomenon that we prove remains true on $\R^d$.
  Typically, only on a compact manifold does $u/|u|$ belong to 
  $L^p_x$ for finite $p$, so the nonlinear term is harder to control
  in the $\R^d$ case. 
\smallbreak

This phenomenon is to be compared with the case of the 
linear damping,
\begin{equation*}
  i\frac{\partial u}{\partial t}+\Delta u= -i\g u.
\end{equation*}
This case is particularly simple, since after the change of unknown function
$v(t,x) = e^{\g t}u(t,x)$, $v$ solves a free Schr\"odinger equation:
its $L^2(M)$ norm does not depend on time, so the $L^2$-norm of $u$
decays exponentially in time. Such a damping term is also used in some physical
models involving an extra interaction nonlinearity, such as a
cubic term; see
e.g. \cite{Fi01} and references therein. Localized linear damping
(replace $\gamma u$ with $a(x)u$) has been considered for control
problems; see e.g. \cite{MaZu94,AlKh07} and references
therein. Stabilization is obtained with an exponential rate in time. 
More recently, nonlinear
damping terms have been considered, but with some homogeneity different
from ours; see e.g. \cite{OhTO09} and references therein. In
\cite{AnSp-p}, the authors consider 
\begin{equation*}
   i\frac{\partial u}{\partial t}+\Delta u= \l |u|^2u -i\g |u|^4
   u,\quad x\in \R^3. 
\end{equation*}
See \cite{AnSp-p} also for references to situations where such a model
is involved.
The nonlinear damping is shown to stabilize the solution, in the sense
that finite time blow-up (which may occur if $\l<0$ and $\g=0$) is
prevented by the damping term ($\g>0$). It can be inferred that the
$L^2$-norm of $u$ goes to zero as time goes to infinity, but probably
more slowly than in the case of a linear damping. Roughly speaking,
the damping is strong only where $u$ is large, so it is less and less
strong as $u$ goes to zero.  
\smallbreak
The damping term present in \eqref{nls} arises in Mechanics (a case
where $u$ is real-valued): in the
case $\alpha=1$, it is 
referred to as \emph{Coulomb friction}. Its effects have been studied
in \cite{AdAtCa06} in the case of ordinary differential equations, and in
\cite{BaCaDi07} in the case of a wave equation. The intermediary case
$0<\alpha<1$ has been studied in the ordinary differential equations
case in \cite{DiLi01,DiLi02,AmDi03}, and the damping term is then
called \emph{strong friction}. As in the case of \eqref{nls}, 
the model one has in mind to understand the dynamics of the equation
is the ordinary differential equation obtained by dropping the
Laplacian in \eqref{nls}:
\begin{equation}
  \label{eq:ode}
  \frac{d u}{d t} = -\g\frac{u}{|u|^\alpha}. 
\end{equation}
Multiplying the above equation by the conjugate of $u$, and setting
$y(t)=|u(t)|^2$, \eqref{eq:ode} yields
\begin{equation*}
  \frac{d y}{dt}= -2\g y^{1-\alpha/2}, 
\end{equation*}
an equation which can be solved explicitly: so long as $y\ge 0$,
\begin{equation*}
  y(t) = \(y(0)^{\alpha/2}-\alpha\g t\)^{2/\alpha}.
\end{equation*}
Therefore, $y$ (hence $u$) becomes zero at time 
\begin{equation*}
  t_c= \frac{|u(0)|^\alpha}{\alpha\g}. 
\end{equation*}
In this paper, we prove a similar phenomenon for weak solutions to
\eqref{nls}. Before stating our main result, we have to specify the
notion of weak solution, especially in the case 
$\alpha=1$, where the right hand side in
\eqref{nls} does not make sense if $u(t,x)=0$.

\begin{definition}[Weak solution, case $0<\alpha<1$]
Suppose $0<\alpha <1$. A (global) weak solution to
  \eqref{nls}  is a function $u \in {\mathcal C}(\R_+;L^2(M))\cap
  L^\infty(\R_+; H^1(M))$ %\cap {\mathcal C}_w (\R_+; H^1(M))$
  solving \eqref{nls} in 
  ${\mathcal D}'(\R_+^*\times M)$.
\end{definition}
\begin{definition}[Weak solution, case $\alpha=1$]
  Suppose $\alpha =1$. A (global) weak solution to
  \eqref{nls}  is a function $u \in {\mathcal C}(\R_+;L^2(M))\cap
  L^\infty(\R_+; H^1(M))$ %\cap {\mathcal C}_w (\R_+; H^1(M))$ 
solving
  \begin{equation*}
  i\frac{\partial u}{\partial t}+\Delta u=  -i\g F  
  \end{equation*}
in ${\mathcal D}'(\R_+^*\times M)$, where $F$ is such that
\begin{equation*}
 \|F\|_{L^\infty(\R_+\times M )} \le 1,\quad \text{and}\quad
 F=\frac{u}{|u|} \text{ if } u\neq 0.
\end{equation*}
\end{definition}
Our main results are as follows.

\begin{theorem}\label{theo:cauchy}
Let $d\ge 1$, $u_0\in H^1(M)$, $\g>0$ and $0<\alpha\le 1$. Then
\eqref{nls} has a unique, global weak solution. 
In addition, it satisfies the \emph{a priori} estimate:
\begin{equation*}
  \|u\|_{L^\infty(\R_+,H^1(M))}\le \|u_0\|_{H^1(M)}. 
\end{equation*}
\end{theorem}
\begin{remark}[Uniqueness]
  Since the nonlinearity in \eqref{nls} is not Lipschitzean,
  uniqueness does not come from completely standard arguments. Note
  that since we consider complex-valued functions, the monotonicity
  arguments invoked in \cite{BaCaDi07} (after \cite{Br72}) cannot be
  used. Uniqueness relies in a crucial manner on the dissipation
  associated to the equation. 
\end{remark}
In the multidimensional setting, our argument to prove finite time
stabilization requires some extra regularity:
\begin{theorem}\label{theo:persist}
Let $d\le 3$, $u_0\in H^2(M)$, $\g>0$ and $0<\alpha\le 1$. Then the
solution of
\eqref{nls} belongs to
$L^\infty(\R_+,H^2(M))$. In addition, there exists $C$, depending only
on $\|u_0\|_{H^2(M)}$, $M$ and $\g$, such that:
\begin{equation*}
  \|u\|_{L^\infty(\R_+,H^2(M))}\le C. 
\end{equation*}\end{theorem}
We will see that for such weak solutions, a complete dissipation
occurs in finite time:
\begin{theorem}\label{theo:vanish}
Let $d\le 3$, $u_0\in H^1(M)$, $\g>0$ and $0<\alpha\le 1$. If
$d=2,3$, suppose in addition that $u_0\in H^2(M)$. Then there exists
$T>0$ such that the (unique) weak solution to \eqref{nls} satisfies
\begin{equation*}
\text{for every }\quad t\ge T,\quad u(t,x)=0,\quad \text {for almost
  every } x\in M.  
\end{equation*}
\end{theorem}
\begin{remark}
  With the results of \cite{MaZu94,AmDi03} in mind, it would seem
  interesting to consider an equation of the form
\begin{equation*}
i\frac{\partial u}{\partial t}+\Delta u= -ia(x)\frac{u}{|u|^\alpha}
\end{equation*}
with $a\ge 0$, to stabilize $u$ in finite time on the set $\{a>0\}$
(one may think of $a$ as an indicator function). 
\end{remark}
As a corollary to our approach, we can prove the same phenomenon for
the ``usual'' nonlinear Schr\"odinger equation perturbed by the
damping term that we consider in this paper, provided 
  that $d=1$ and
that no finite
time blow-up occurs without damping:
\begin{corollary}\label{cor:stabnls}
  Let $d=1$, $u_0\in H^1(M)$, $\g,\si>0$, $0<\alpha\le 1$ and $\l\in
  \R$. Assume in addition $\si<2$ if $\l<0$. Then there exists
$T>0$ such that the (unique) weak solution to 
\begin{equation}\label{eq:nls2}
  i\frac{\d u}{\d t}+\Delta u = \l |u|^{2\si}u -i\g
  \frac{u}{|u|^\alpha}\quad ;\quad u_{\mid t=0}=u_0
\end{equation}
satisfies: for every $t\ge T$, $u(t,x)=0$ for almost every $x\in M$.
  \end{corollary}
  \begin{remark}
    The notion of weak solution for \eqref{eq:nls2} is easily adapted,
    as well as the proof of uniqueness,
    since for $d=1$, $H^1(M)\hookrightarrow 
    L^\infty(M)$. Proving the analogue of Corollary~\ref{cor:stabnls}
    in a multi-dimensional framework, or in cases where finite time
    blow-up occurs when $\g=0$ (e.g., $d=1$, $\l<0$ and $\si\ge 2$),
    seems to be an interesting open question. Note however that the
    results in \cite{MeRa04} and the fact that the present damping is
    sublinear suggest that there is no universal conclusion in that
    case: in the competition between finite time blow-up
    and dissipation, either of the two effects may win. 
  \end{remark}

\section{Existence results}
\label{sec:exist}

To prove the existence part of Theorem~\ref{theo:cauchy}, we first
regularize the nonlinearity to construct a mild solution, and then pass
to the limit. This procedure allows us to prove
Theorem~\ref{theo:persist} too. Uniqueness is established in
Section~\ref{sec:unique}.  

\subsection{Construction of an approximating sequence}
\label{sec:sequence}

In order to construct a solution of \eqref{nls} as in
Theorem~\ref{theo:cauchy}, we solve first, for $\delta >0$, the equation 
\begin{equation}\label{nlsdel}
i\frac{\partial u^\delta}{\partial t}+\Delta u^\delta=
f_\delta(u^\delta):=-i\g\frac{u^\delta}{(|u^\delta|^2+
\delta)^{\alpha/2}}.  
\end{equation}

\begin{proposition}\label{cauchydel}
Let $\delta>0$, $u_0\in L^2( M)$. There exists a unique $u^\delta\in
\mathcal{C}(\R_+,L^2( M))$ such that 
\begin{equation*}
  u^\delta(t)=e^{it\Delta}u_0-i\int_0^te^{i(t-\tau)\Delta}
f_\delta(u^\delta(\tau))d\tau,\quad 
t\in\R_+.
\end{equation*}
If moreover $u_0\in H^s( M)$ for some $s\ge 0$, then
$u\in\mathcal{C}(\R_+,H^s( M))\cap\mathcal{C}^1(\R_+,H^{s-2}( M))$. The 
flow map
\begin{equation*}
  \begin{array}{rcl}
H^s( M)&\to & \mathcal{C}(\R_+,H^s( M))\\
u_0 &\mapsto & u^\delta
\end{array}
\end{equation*}
is continuous. If $u_0\in H^1( M)$, for every $t\ge 0$, we have 
\begin{equation}\label{massdel}
\|u^\delta(t)\|_{L^2( M)}^2+2\g\int_0^t\int_{
  M}\frac{|u^\delta(\tau)|^2}{(|u^\delta(\tau)|^2+
\delta)^{\alpha/2}}dxd\tau=\|u_0\|_{L^2(
  M)}^2, 
\end{equation}
\begin{equation}\label{energydel}
  \begin{aligned}
   &\|\nabla  u^\delta(t)\|_{L^2( M)}^2-\|\nabla u_0\|_{L^2( M)}^2=\\
& -2\g\int_0^t\int_{
    M}\frac{\delta|\nabla
    u^\delta|^2+(1-\alpha)\lvert \RE(\overline{u^\delta}\nabla
    u^\delta)\rvert^2+\lvert\IM(\overline{u^\delta}\nabla
    u^\delta)\rvert^2}{(|u^\delta|^2+\delta)^{\alpha/2+1}}(\tau,x)dxd\tau. 
  \end{aligned}
\end{equation}
\end{proposition}

\begin{proof} Since $f_\delta\in\mathcal{C}^\infty(\C,\C)$ is globally
  Lipschitzean, the global well-posedness in $H^s$ and the continuity
  of the flow 
  map are well known, and follow from the standard fixed point
  argument and Gronwall lemma (see e.g. \cite{CaHa98}). The identity
  \eqref{massdel} is first obtained for $u_0\in  
  H^s$ with $s$ large, multiplying  \eqref{nlsdel} by the conjugate
  of 
  $u^\delta$, taking the imaginary part and integrating in space and
  time. The identity \eqref{massdel} is then obtained for $u_0\in L^2( M)$
  thanks to the continuity of 
  the flow map and the density of $H^s$ in $L^2$. Similarly,
  \eqref{energydel} is first obtained for $u_0\in
  H^s$ with $s$ large, multiplying \eqref{nlsdel} by
  $\partial_t\overline{u^\delta}$, taking the imaginary part, and
  integrating. Alternatively, \eqref{energydel} can be obtained
  formally by applying
  the operator $\nabla$ to \eqref{nlsdel}, multiplying the result by
  $\nabla \overline{u^\delta}$, taking the imaginary part, and
  integrating. 
\end{proof}

To prove Theorem~\ref{theo:persist}, we will also use the following
result:
\begin{proposition}\label{prop:cauchydel2}
Let $\delta>0$, $u_0\in H^2( M)$, and $u^\delta\in
\mathcal{C}(\R_+,H^2( M))\cap \mathcal{C}^1(\R_+,L^2(M))$ 
be as in Proposition~\ref{cauchydel}. Then for every
$t\ge 0$, we have  
\begin{equation}
  \label{eq:evoldtu}
 \begin{aligned}
   &\|\d_t  u^\delta(t)\|_{L^2( M)}^2-\|\d_t u^\delta(0)\|_{L^2( M)}^2=\\
& -2\g\int_0^t\int_{
    M}\frac{\delta|\d_t
    u^\delta|^2+(1-\alpha)|\RE(\overline{u^\delta}\d_t
    u^\delta)|^2+|\IM(\overline{u^\delta}\d_t
    u^\delta)|^2}{(|u^\delta|^2+\delta)^{\alpha/2+1}}(\tau,x)dxd\tau. 
  \end{aligned}  
\end{equation}
\end{proposition}
\begin{proof}
   Since 
 $u^\delta \in \mathcal{C}(\R_+,H^2(M))$ from
   Proposition~\ref{cauchydel}, and   
\begin{equation*}
  \left\lvert \frac{u^\delta}{(|u^\delta|^2+
\delta)^{\alpha/2}} \right\rvert \le
\frac{|u^\delta|}{\delta^{\alpha/2}},
\end{equation*}
Equation~\eqref{nlsdel} implies $\d_t u^\delta \in
\mathcal{C}(\R_+,L^2(M))$, and $\d_t u^\delta$ solves
\begin{equation*}
\(i\frac{\partial }{\partial t}+\Delta \)\frac{\d u^\delta}{\d t}=
-i\g \frac{\d_t u^\delta}{ (|u^\delta|^2+
\delta)^{\alpha/2}} +i\g \frac{\alpha}{2}\frac{u^\delta}{(|u^\delta|^2+
\delta)^{\alpha/2+1}} \d_t |u^\delta|^2.
\end{equation*}
On a formal level, we infer
\begin{align*}
  \frac{d}{dt}\|\d_t u^\delta\|_{L^2(M)}^2 &= 2\RE\int_M \frac{\d
    \overline{u^\delta}}{\d t}\frac{\d^2 u^\delta}{\d t^2}\\
&=- 2\g \int_M \frac{|\d_t u^\delta|^2}{(|u^\delta|^2+
  \delta)^{\alpha/2}} +2\alpha\g \int_M 
\frac{\( \RE (\overline{u^\delta}\d_t u^\delta)
  \)^2}{(|u^\delta|^2+ \delta)^{\alpha/2+1}}\\
&= -2\g \int_M (|u^\delta|^2+
  \delta)\frac{|\d_t u^\delta|^2}{(|u^\delta|^2+
  \delta)^{\alpha/2+1 }} +2\alpha\g \int_M 
\frac{\( \RE (\overline{u^\delta}\d_t u^\delta)
  \)^2}{(|u^\delta|^2+ \delta)^{\alpha/2+1}}. 
\end{align*}
The identity \eqref{eq:evoldtu} then follows by decomposing
\begin{equation*}
  |u^\delta|^2|\d_t u^\delta|^2 = \(\RE (\overline{u^\delta}\d_t u^\delta)
  \)^2 + \(\IM (\overline{u^\delta}\d_t u^\delta)  \)^2.
\end{equation*}
The result follows from the same arguments as in the proof of
Proposition~\ref{cauchydel}. 
\end{proof}
\begin{remark}
  Finite time stabilization must not be expected to occur in
  \eqref{nlsdel} (for $\delta>0$). The corresponding toy model is the
  ordinary differential equation
  \begin{equation*}
    \frac{d u^\delta}{dt} = - \g
    \frac{u^\delta}{\(|u^\delta|^2+\delta\)^{\alpha/2}}\quad ;\quad
  u^\delta(0)=u_0. 
  \end{equation*}
Setting again $y_\delta=|u^\delta|^2$, it now solves
\begin{equation*}
  \frac{dy_\delta}{dt} = -2\g
  \frac{y_\delta}{\(y_\delta+\delta\)^{\alpha/2}}\quad ;\quad
  y_\delta(0)=|u_0|^2.  
\end{equation*}
Now since 
\begin{equation*}
  \int_\tau^1 \frac{\(y+\delta\)^{\alpha/2}}{y}dy
\end{equation*}
diverges logarithmically as $\tau\to 0^+$, $y_\delta$ decays
exponentially in time. A change of time variable shows that there
exists $C$ independent of $\delta\in ]0,1]$ such that
\begin{equation*}
  y_\delta(t)\le C e^{-Ct/\delta^{\alpha/2}}, \quad \forall t\ge 0. 
\end{equation*}
The exponential decay is stronger and stronger as $\delta$ goes to
zero. The example discussed in the introduction shows that in the
limit $\delta \to 0$, this exponential decay becomes a finite time
arrest. Proving Theorem~\ref{theo:vanish} somehow amounts to showing the same
phenomenon in a PDE setting.  
\end{remark}

\subsection{Convergence of the approximation}
\label{sec:conv}

The fact that the approximating sequence $(u^\delta)_\delta$ converges
to a weak solution of \eqref{nls} follows essentially from the same
arguments as in \cite{GiVe85}. 
\smallbreak

A straightforward
consequence from \eqref{massdel} and \eqref{energydel} is that for 
$u_0\in H^1( M)$ fixed, the sequence $(u^\delta)_{0<\delta\le 1}$ is
uniformly bounded in $L^\infty(\R_+,H^1( M))\cap
L^{2-\alpha}(\R_+\times M)$. Since $L^\infty(\R_+,H^1( M))$ is the dual
of $L^1(\R_+,H^{-1}( M))$, we deduce the existence of
$u\in L^\infty(\R_+,H^1( M))$ and of a
subsequence $u^{\delta_n}$ such that
\begin{equation}\label{convf*}
u^{\delta_n}\rightharpoonup u,\quad  \text{in }w*\ L^\infty(\R_+,H^1( M)),
\end{equation}
with, in view of \eqref{massdel} and \eqref{energydel},
\begin{equation*}
  \|u\|_{L^\infty(\R_+,H^1( M))}\le \|u_0\|_{H^1(M)}.
\end{equation*}
Moreover,
$\frac{u^\delta}{(|u^\delta|^2+\delta)^{\alpha/2}}$ is uniformly bounded
in $L^\infty(\R_+,L^{\frac{2}{1-\alpha}}( M))$ (with $2/(1-\alpha)=\infty$
if $\alpha=1$), such that up to the
extraction of an 
other subsequence, there is $F\in
L^\infty(\R_+,L^{\frac{2}{1-\alpha}}( M))$ 
such that
\begin{equation}\label{convf*F}
\frac{u^{\delta_n}}{(|u^{\delta_n}|^2+\delta_n)^{\alpha/2}}
\rightharpoonup F,\quad \text{in } w*\ L^\infty(\R_+,L^{\frac{2}{1-\alpha}}( M)).
\end{equation}
Moreover, $\|F\|_{L^\infty(\R_+,L^{\frac{2}{1-\alpha}}( M))}\le \|u_0\|_{L^2(M)}^{1-\alpha}$.
Let $\theta\in \mathcal{C}_c^\infty(\R_+^*\times M)$. Then
\begin{align*}
\<-i\g\frac{u^{\delta_n}}{(|u^{\delta_n}|^2+\delta_n)^{\alpha/2}},
  \theta\>&=\<i\frac{\partial 
  u^{\delta_n}}{\partial t}+\Delta
u^{\delta_n},\theta\>=\<u^{\delta_n},-i\frac{\partial
  \theta}{\partial t}+\Delta \theta\>\\
&\Tend n \infty \<u,-i\frac{\partial
  \theta}{\partial t}+\Delta \theta\>=\<i\frac{\partial
  u}{\partial t}+\Delta u,\theta\>,
\end{align*}
where $\<\cdot,\cdot\>$ stands for the distribution bracket on
$\R_+^*\times M$. Thus, we deduce 
$$i\frac{\partial u}{\partial t}+\Delta u=-i\g F,\quad \text{ in }
\mathcal{D}'(\R_+^*\times M).$$
We next show that $F= u/|u|^\alpha$ where the right hand side is well
defined, that is if $\alpha<1$, or $\alpha=1$ and $u\neq 0$. We first
suppose that $u_0\in 
H^s( M)$ with $s$ large. Let us fix $t'\in \R_+$ and $\delta>0$. Thanks to
\eqref{massdel}, we infer, for any $t\in \R_+$,
\begin{align}
\frac{d}{dt}\|u^\delta(t)-u^\delta(t')\|_{L^2}^2 & \le 
\frac{d}{dt}\big(-2\RE
\left(u^\delta(t)|u^\delta(t')\right)\big)\nonumber\\
&=-2\RE\left(i\Delta u^\delta(t)-\frac{\g
    u^\delta(t)}{(|u^\delta(t)|^2+\delta)^{\alpha/2}}\Big|u^\delta(t')\right), 
\end{align}
where $(\cdot|\cdot)$ denotes the scalar product in $L^2(M)$. By
integration, we deduce
\begin{equation}\label{cont}
  \begin{aligned}
   \|u^\delta(t)-u^\delta(t')\|_{L^2( M)}^2\le 2|t-t'|\Big(& \|\Delta
u^\delta\|_{L^\infty(\R_+,H^{-1}( M))}
\|u^\delta\|_{L^\infty(\R_+,H^{1}( M))}\\
&+
\g\|u^\delta\|^{2-\alpha}_{L^\infty(\R_+,L^{2-\alpha}( M))}\Big).
  \end{aligned}
\end{equation}
From the continuity of the flow map $H^1\ni u_0\mapsto u^\delta\in
\mathcal{C}(\R_+,H^1)$ in Proposition~\ref{cauchydel}, we deduce that
\eqref{cont} also holds if we only have $u_0\in H^1( M)$. Next,
since $(u^\delta)_{0<\delta\le 1}$ is uniformly bounded in
$L^\infty(\R_+,H^1( M))$ and $ M$ is compact, 
\eqref{cont} gives the existence of a positive constant $C$ such that
for every $t,t'\in \R_+$,
$$\|u^\delta(t)-u^\delta(t')\|_{L^2( M)}\le C|t-t'|^{1/2}.$$
In particular, for any $T>0$, $(u^\delta)_{0<\delta\le 1}$ is a
bounded sequence in $\mathcal{C}([0,T],L^2( M))$ which is uniformly
equicontinuous from $[0,T]$ to $L^2( M)$. Moreover, the compactness of
the embedding $H^1(M)\subset L^2(M)$ ensures that for every $t\in
[0,T]$, the set $\{u^\delta(t)|\delta\in (0,1]\}$ is relatively compact in $L^2(M)$. As a result, Arzel\`a--Ascoli
Theorem ensures that $(u^{\delta_n})_n$ is relatively compact in
$\mathcal{C}([0,T],L^2( M))$. On the other hand, we already know
from \eqref{convf*} that 
$$u^{\delta_n}\rightharpoonup u\quad \text{in } w*\ L^\infty(\R_+,L^2( M)).$$
Therefore, we infer that $u$ is the unique accumulation point of the sequence
$(u^{\delta_n})_n$ in $\mathcal{C}([0,T], L^2( M))$. Thus
$$u^{\delta_n}\to u\quad \text{in }\mathcal{C}([0,T],L^2( M)),$$
which implies in particular $u\in \mathcal{C}([0,T],L^2( M))$
as well as $u(0)=u^{\delta_n}(0)=u_0$. This is true for any $T>0$,
therefore 
$$u\in \mathcal{C}(\R_+,L^2( M)).$$
Finally, up to the extraction of an other subsequence, 
$u^{\delta_n}(t,x)\to u(t,x)$ for almost every $(t,x)\in
\R_+\times M$. Therefore, for almost every $(t,x)\in
\R_+\times M$ such that $u(t,x)\neq 0$, we have
$$\frac{u^{\delta_n}}{(|u^{\delta_n}|^2+\delta_n)^{\alpha/2}}(t,x)\to
\frac{u}{|u|^\alpha}(t,x).$$
By comparison with \eqref{convf*F}, we deduce that up to a change of
$F$ on a set with zero measure,
$$F(t,x)=\frac{u}{|u|^\alpha}(t,x)\quad\text{(only if }u(t,x)\neq
0\text{ in the case }\alpha=1\text{)},$$ 
which completes the proof of the existence part 
of Theorem~\ref{theo:cauchy}.

\subsection{Proof of Theorem~\ref{theo:persist}}
\label{sec:persist}

To prove Theorem~\ref{theo:persist}, we resume the idea due to T.~Kato
\cite{Kato87} (see also \cite{CazCourant}), based on the
general idea for 
Schr\"odinger equation, that 
two space derivative cost the same as one time derivative. 
\smallbreak

The time derivative of $u^\delta$ at time $t=0$ is given by the
equation: from \eqref{nlsdel}, 
\begin{equation*}
  \frac{\d u^\delta}{\d t}(0) = i\Delta u^\delta -
  if_\delta\(u^\delta\)\Big|_{t=0} = i\Delta u_0 -\g
  \frac{u_0}{\(|u_0|^2+\delta\)^{\alpha/2}}. 
\end{equation*}
We infer, for $u_0\in H^2(M)$,
\begin{align*}
  \left\lVert \frac{\d u^\delta}{\d t}(0)\right\rVert_{L^2(M)} &\le
  \|\Delta u_0\|_{L^2(M)} + \g \left\lVert \lvert
      u_0\rvert^{1-\alpha}\right\rVert_{L^2(M)} \\
&\le \|\Delta
    u_0\|_{L^2(M)} + C(\alpha,M)\|u_0\|_{L^2(M)}^{1-\alpha},
\end{align*}
since $M$ is compact. By Proposition~\ref{prop:cauchydel2}, the
$L^2$-norm of $\d_t u^\delta$ is a non-increasing function of time,
and there exists $C$ such that 
\begin{equation*}
  \left\lVert \frac{\d u^\delta}{\d t}(t)\right\rVert_{L^2(M)}\le
  C,\quad \forall t\in \R_+,\ \forall \delta \in ]0,1]. 
\end{equation*}
Using \eqref{nlsdel} again, we infer
\begin{equation*}
  \|\Delta u^\delta(t)\|_{L^2(M)}\le C + \g \left\lVert \lvert
      u(t)\rvert^{1-\alpha}\right\rVert_{L^2(M)}\le \widetilde C,
\end{equation*}
where $\widetilde C$ is independent of $t\in \R_+$ and
$\delta\in]0,1]$, since $M$ is compact and since \eqref{massdel} implies
\begin{equation*}
  \left\lVert u^\delta(t)\right\rVert_{L^2(M)}\le 
\left\lVert u_0\right\rVert_{L^2(M)}
  ,\quad \forall t\in \R_+,\ \forall \delta \in ]0,1]. 
\end{equation*}
Therefore, there exists $C$ depending only on $\|u_0\|_{H^2(M)}$, $M$
and $\g$ such that 
\begin{equation*}
   \|\Delta u^\delta\|_{L^\infty(\R_+,L^2(M))}\le C.
\end{equation*}
By Fatou's Lemma, we conclude that if $u_0\in H^2(M)$, then the weak
solution $u$ satisfies $u\in L^\infty(\R_+,H^2(M))$.

\section{Uniqueness}
\label{sec:unique}
We start the proof of uniqueness in Theorem~\ref{theo:cauchy} with
the following lemma.

\begin{lemma}\label{lem:young}
  Let $\alpha\in ]0,1]$. For all $z_1,z_2\in \C$,
  \begin{equation*}
    \RE \( \(\frac{z_1}{|z_1|^{\alpha}}-
    \frac{z_2}{|z_2|^{\alpha}}\)\(\overline{z_1-z_2}\)\)\ge 0.
  \end{equation*}
\end{lemma}
\begin{proof}
  Pick $\rho_1,\rho_2\ge 0$ and $\theta_1,\theta_2\in [0,2\pi[$ such
  that
  \begin{equation*}
    z_j=\rho_je^{i\theta_j},\quad j=1,2.
  \end{equation*}
We write
\begin{align*}
 \RE \( \(\frac{z_1}{|z_1|^{\alpha}}-
    \frac{z_2}{|z_2|^{\alpha}}\)\(\overline{z_1-z_2}\)\)&=
    \rho_1^{2-\alpha} +\rho_2^{2-\alpha} -\rho_1^{1-\alpha} \rho_2
    \cos\(\theta_1-\theta_2\)\\
&\quad -\rho_1 \rho_2^{1-\alpha}
    \cos\(\theta_1-\theta_2\) \\
&\ge \rho_1^{2-\alpha} +\rho_2^{2-\alpha} -\rho_1^{1-\alpha} \rho_2
    -\rho_1 \rho_2^{1-\alpha}. 
\end{align*}
For $\alpha=1$, the conclusion is then obvious. For $\alpha\in ]0,1[$,
we use Young's inequality:
\begin{equation*}
  ab\le \frac{a^p}{p}+\frac{b^{p'}}{p'},\quad \forall a,b\ge 0, \
  \forall p\in ]1,\infty[.
\end{equation*}
With $p=\frac{2-\alpha}{1-\alpha}$ and $p=2-\alpha$, respectively, we infer
\begin{equation*}
  \rho_1^{1-\alpha} \rho_2 \le
  \(\frac{1-\alpha}{2-\alpha}\)\rho_1^{2-\alpha} +
  \frac{\rho_2^{2-\alpha}}{2-\alpha} \quad ;\quad \rho_1 \rho_2^{1-\alpha} \le
    \frac{\rho_1^{2-\alpha}}{2-\alpha}+
\(\frac{1-\alpha}{2-\alpha}\)\rho_2^{2-\alpha} .
\end{equation*}
The lemma follows. 
\end{proof}

Next, we prove the following energy estimate, which is shown to hold
for any solution to (\ref{nls}).

\begin{proposition}\label{prop:mass}
Let $d\ge 1$. Let $u_0,v_0\in H^1(M)$ and $u,v\in
\mathcal{C}(\R_+,L^2( M))\cap L^\infty(\R_+,H^1( M))$ be two solutions
of \eqref{nls} with initial data $u(0)=u_0$ and $v(0)=v_0$
respectively. Then the map $m_{u,v}:t\mapsto \|(u-v)(t)\|_{L^2(M)}^2$ is
differentiable everywhere on $\R_+$,
$m_{u,v}'\in L^1_{\rm loc}(\R_+)$ and for
every $t\in\R_+$,
\begin{equation}\label{eq:massuv}
\frac{d}{dt}\|(u-v)(t)\|_{L^2( M)}^2+2\g\int_M\RE\left(\left(\frac{u(t)}{|u(t)|^{\alpha}}-\frac{v(t)}{|v(t)|^{\alpha}}\right) \overline{u(t)-v(t)}\right)dx=0
\end{equation}
In particular, if $v$ is taken to be the trivial solution $v\equiv 0$,
we have for any solution of \eqref{nls}:
 \begin{equation}\label{eq:mass}
\frac{d}{dt}\|u(t)\|_{L^2( M)}^2+2\g\|u(t)\|^{2-\alpha}_{L^{2-\alpha}(
  M)}=0. 
\end{equation} 
\end{proposition}

\begin{proof}
First, notice that if $u\in \mathcal{C}(\R_+,L^2( M))\cap
L^\infty(\R_+,H^1( M))$, then for every $t\in \R_+$, $u(t)\in
H^1(M)$. This is so because  $u$ is weakly continuous in time, with
values in $H^1(M)$: $u\in {\mathcal
    C}_w(\R_+,H^1(M))$. 
Indeed, if $t\ge 0$ is fixed, since $u\in
L^\infty(\R_+,H^1(M))$, there exists a sequence $t_n\to t$ such that
for every $n$, $\|u(t_n)\|_{H^1(M)}\le \|u\|_{L^\infty(\R_+,H^1(M))}$. Then, for
every $\phi\in H^1(M)$ and $j\in\{1,\cdots,d\}$, 
\begin{align}
\left<\d_j u(t),\phi\right>_{H^{-1},H^1} & =  -\left<u(t),\d_j
\phi\right>_{L^2,L^2}=-\underset{n\to\infty}{\lim}\left<u(t_n),\d_j
\phi\right>_{L^2,L^2}\nonumber\\
& = \underset{n\to\infty}{\lim}\left<\d_j u(t_n),\phi\right>_{L^2,L^2},\nonumber
\end{align}
thus
\begin{equation*}
\left|\left<\d_j u(t),\phi\right>_{H^{-1},H^1}\right|
\le\|u\|_{L^\infty(\R_+,H^1(M))}\|\phi\|_{L^2(M)},
\end{equation*}
which implies that $\nabla u(t)\in L^2(M)^d$. As a result, for every
$t\in \R_+$,
\begin{equation*}
\frac{\d u}{\d t}(t)=i\Delta u(t)-\g\frac{u(t)}{|u(t)|^\alpha}\in H^{-1}(M). 
\end{equation*}
Then, if $u,v$ are as in the statement of Proposition~\ref{prop:mass},
$m_{u,v}$ is 
differentiable everywhere on $\R_+$, and
\begin{align}
m_{u,v}'(t) & =  2\RE\left<\frac{\partial (u-v)}{\partial
  t}(t),\overline{(u-v)(t)}\right>_{H^{-1}( M),H^1( M)}\nonumber\\ 
& = 
2\RE\left<i\Delta(u-v)(t)-\g\left(\frac{u(t)}{|u(t)|^\alpha}-\frac{v(t)}{|v(t)|^\alpha}\right),\overline{(u-v)(t)}\right>_{H^{-1}(M),H^1(M)}\nonumber\\  
&=
-2\g\int_M\RE\left(\left(\frac{u(t)}{|u(t)|^\alpha}-\frac{v(t)}{|v(t)|^\alpha}\right)\overline{(u-v)(t)}\right)dx.
\end{align}
Since $u,v\in L^\infty(\R_+,L^2(M))$ and $L^2(M)\subset
L^{2-\alpha}(M)$ by compactness of $M$, we deduce from the
Cauchy-Schwarz inequality that $m_{u,v}'\in L^\infty(\R_+)\subset
L^1_{\rm loc}(\R_+)$. 
\end{proof}
It follows from Proposition~\ref{prop:mass},
Lemma~\ref{lem:young} and the Fundamental Theorem of
Calculus, that if $u$ and $v$ are chosen as in
Proposition~\ref{prop:mass}, $m_{u,v}(t)=\|(u-v)(t)\|_{L^2(M)}^2$ is
non-increasing on $\R_+$. The uniqueness part of Theorem~\ref{theo:cauchy}
follows, choosing two solutions $u$ and $v$ of (\ref{nls}) with the
same initial datum $u(0)=v(0)$.

\section{Finite time stabilization: proof of Theorem~\ref{theo:vanish}}
\label{sec:dissip}

We next show that under the assumptions of Theorem~\ref{theo:vanish},
$u$ vanishes 
in finite time. The proof relies on a Nash type inequality:
\begin{lemma}\label{lem:nash}
  Let $(M,g)$ be a smooth compact Riemannian manifold, of dimension
  $d$, and $\alpha\in ]0,1]$. There exists $C>0$ such that
 \begin{align}
    \|f\|_{L^2(M)}^{\alpha d + 4-2\alpha}&\le C
    \(\|f\|^{2-\alpha}_{L^{2-\alpha}(M)}\)^{2}\|f\|_{H^1(M)}^{\alpha
      d},\quad \forall f\in 
    H^1(M).\label{eq:nash1}\\ 
\|f\|_{L^2(M)}^{\alpha d + 8-4\alpha }&\le C
    \(\|f\|^{2-\alpha}_{L^{2-\alpha}(M)}\)^{4}\|f\|_{H^2(M)}^{\alpha
      d},\quad \forall f\in
    H^2(M).\label{eq:nash2}
  \end{align} 
\end{lemma}
\begin{proof}
 As it is
standard in geometry, inequalities valid on $\R^d$ are
easily transported to the case of compact manifolds (see
e.g. \cite{He99}).  Since $M$ is compact, $M$ can be covered by a
finite number of 
  charts
  \begin{equation*}
    (\Om_n,\varphi_n)_{1\le n\le N}
  \end{equation*}
such that for any $n$, the components $g_{ij}^n$ of $g$ in
$(\Om_n,\varphi_n)$ satisfy 
\begin{equation*}
  \frac{1}{2}\delta_{ij}\le g_{ij}^n \le 2\delta_{ij}
\end{equation*}
as bilinear forms. Let $(\eta_n)_{1\le n\le N}$ be a smooth partition
of unity subordinate to the covering $(\Om_n)_{1\le n\le N}$. For any
$f\in C^\infty(M)$ and any $n$, we have
\begin{align*}
  \int_M |\eta_n f|^p &\le 2^{d/2} \int_{\R^d}\left\lvert \(\eta_n
    f\)\circ \varphi_n^{-1}(x)\right\rvert^p dx,\quad 1\le p\le 2,\\
\int_M \left\lvert \nabla\(\eta_n f\)\right\rvert^2 &\ge 2^{-d/2}
\int_{\R^d}\left\lvert \nabla\(\(\eta_n 
    f\)\circ \varphi_n^{-1}\)(x)\right\rvert^2 dx,\\
\int_M \left\lvert \Delta\(\eta_n f\)\right\rvert^2 &\ge 2^{-d/2}
\int_{\R^d}\left\lvert \Delta\(\(\eta_n 
    f\)\circ \varphi_n^{-1}\)(x)\right\rvert^2 dx.
\end{align*}
The lemma follows from inequalities on $\R^d$, adapted from the
Nash inequality \cite{Na58}: for all $\alpha\in ]0,1]$ and all $s>0$,
there exists $C=C(\alpha,s)$ such that 
 \begin{equation}\label{eq:nashgen}
    \left\lVert g\right\rVert_{L^2(\R^d)}^{\alpha d + 2s(2-\alpha)}\le C
    \(\|g\|^{2-\alpha}_{L^{2-\alpha}(\R^d)}\)^{2s}\|g\|_{\dot H^s(\R^d)}^{\alpha
      d}, \quad \forall g\in H^s(\R^d)\cap L^{2-\alpha}(\R^d),
  \end{equation} 
where $\dot H^s(\R^d)$ denotes the homogeneous Sobolev space. Note
that for $s=1$ and $s=2$, we recover the numerology of
\eqref{eq:nash1} and \eqref{eq:nash2}, respectively. To prove
\eqref{eq:nashgen}, use Plancherel formula and decompose the frequency
space: for $R>0$, write
\begin{equation*}
  \|g\|_{L^2(\R^d)}\lesssim \|\widehat g\|_{L^2(|\xi| \le R)} +
  \|\widehat g\|_{L^2(|\xi|> R)} \lesssim R^{d/q}\|\widehat
  g\|_{L^p(\R^d)} + R^{-s}\left\lVert \lvert \xi\rvert^s \widehat
    g\right\rVert_{L^2(\R^d)}, 
\end{equation*}
where $1/2=1/q+1/p$. Choose $p$ so that its H\"older conjugate
exponent is $p'= 2-\alpha\in [1,2[$. Hausdorff--Young inequality
implies
\begin{equation*}
\|g\|_{L^2(\R^d)}\lesssim  R^{d/q} \|g\|_{L^{2-\alpha}(\R^d)}
+R^{-s}\|g\|_{\dot H^s(\R^d)}.  
\end{equation*}
We compute $q=2(2-\alpha)/\alpha$. Optimizing in $R$ yields
\begin{equation*}
  R^{s+ (\alpha d)/(2(2-\alpha))} = \frac{\|g\|_{\dot
      H^s(\R^d)}}{\|g\|_{L^{2-\alpha}(\R^d)}},
\end{equation*}
where we point out that getting the best possible constant is not our
goal. This value of $R$ yields \eqref{eq:nashgen}. 
Lemma~\ref{lem:nash} then follows by
using the chain rule, the fact that $\eta_n$ is smooth on $M$, and
summing over $n$. 
\end{proof}

To prove finite time extinction, we treat separately the cases $d=1$
on the one hand, and $d=2,3$ on the 
other hand. 
In the one-dimensional case, the identity \eqref{eq:mass}
and Nash inequality \eqref{eq:nash1} yield
\begin{equation}\label{eq:massineq1}
 \frac{d}{dt}\|u(t)\|_{L^2( M)}^2 +
\frac{C\g}{\|u(t)\|_{H^1(M)}^{\alpha/2}} 
\|u(t)\|_{L^2(M)}^{2-\alpha/2}\le 0,  
\end{equation}
for some $C>0$ independent of $t$, $\g$ and $u$. 
From Theorem~\ref{theo:cauchy}, we infer
\begin{equation*}
   \frac{d}{dt}\|u(t)\|_{L^2( M)}^2 +
\frac{C\g}{\|u_0\|_{H^1(M)}^{\alpha/2}} 
\|u(t)\|_{L^2(M)}^{2-\alpha/2}\le 0. 
\end{equation*}
By integration, we deduce, as long as $\|u(t)\|_{L^2(M)}$ is not equal
  to zero: 
  \begin{equation*}
    \|u(t)\|_{L^2( M)} \le \(\|u_0\|_{L^2(M)}^{\alpha/2}-
    \frac{C\g}{\|u_0\|_{H^1(M)}^{\alpha/2}}  t\)^{2/\alpha}.
  \end{equation*}
We infer that $\|u(t)\|_{L^2( M)}$ vanishes in finite time, at a time
$$T_v:=\sup\{t\in\R_+| \|u(t)\|_{L^2( M)}\not= 0\},$$
which is bounded from above by
\begin{equation*}
  T_v\le \frac{1}{C\g}\|u_0\|_{L^2(M)}^{\alpha/2}\|u_0\|_{H^1(M)}^{\alpha/2}.
\end{equation*}
Using \eqref{eq:mass} again, and the mere fact that the $L^2$-norm of
$u$ is a non-increasing function of time, we conclude that
$\|u(t)\|_{L^2(M)}=0$ for all $t>T_v$. 
\begin{remark}
  Without the information $u\in L^\infty(\R_+,H^1(M))$, we cannot
  conclude after \eqref{eq:massineq1}, in general. For instance, if we
  have $\|u(t)\|_{H^1(M)}\le C e^{Ct}$, the integration of
  \eqref{eq:massineq1} does not necessarily yield finite 
  time stabilization. 
\end{remark}
\begin{remark}
  Similarly, it might be tempting to first integrate \eqref{eq:mass}
  with respect to time, and then use the fact that $M$ is compact, to
  write
  \begin{equation*}
    \|u(t)\|_{L^{2-\alpha}(M)}^2+C\g \int_0^t
    \|u(\tau)\|_{L^{2-\alpha}(M)}^{2-\alpha}d\tau \lesssim \|u_0\|_{L^2(M)}^2.
  \end{equation*}
However, this inequality does not rule out, e.g., an
exponential decay in time. 
\end{remark}
We see that the key in the above argument is that we have controlled
the term $\|u(t)\|_{L^{2-\alpha}(M)}^{2-\alpha}$ by
$\|u(t)\|_{L^2(M)}^\beta$ for some $\beta<2$, in order to recover the
ODE mechanism presented in the introduction. With the uniform $H^1$
estimate given in Theorem~\ref{theo:cauchy}, \eqref{eq:nash1} yields
such a control provided that 
\begin{equation*}
  \alpha d +4-2\alpha < 4,\text{ that is, if
  }\alpha\(\frac{d}{2}-1\)<0. 
\end{equation*}
Since $\alpha\in ]0,1]$, this is possible if, and only if, $d=1$. 
For $d=2,3$, we therefore use \eqref{eq:nash2} and
Theorem~\ref{theo:persist}. We infer similarly
\begin{equation*}
   \frac{d}{dt}\|u(t)\|_{L^2( M)}^2 +
C\g
\|u(t)\|_{L^2(M)}^{2-(1-d/4)\alpha}\le 0. 
\end{equation*}
Again since $(1-d/4)\alpha>0$, we infer that $\|u(t)\|_{L^2( M)}$
vanishes in finite time $T_v$, with 
\begin{equation*}
  T_v \le \frac{1}{C\g}\|u_0\|_{L^2(M)}^{(1-d/4)\alpha}.
\end{equation*}
Note that unlike in the one-dimensional case, this constant $C$ depends
on $u_0$ (on $\|u_0\|_{H^2(M)}$ only), $M$ and $\g$.

\section{Finite time stabilization: proof of Corollary~\ref{cor:stabnls}}
\label{sec:casnl}

Since Corollary~\ref{cor:stabnls} includes the assumption $d=1$, we
shall be rather brief for the analogue of
Theorem~\ref{theo:cauchy}. 
We can resume the strategy presented in Section~\ref{sec:exist}, to
construct an approximating sequence solution to 
\begin{equation}\label{eq:nlsdel2}
i\frac{\partial u^\delta}{\partial t}+\Delta u^\delta=\l \lvert
u^\delta\rvert^{2\si} u^\delta-i\g\frac{u^\delta}{(|u^\delta|^2+
\delta)^{\alpha/2}}.  
\end{equation}
Since $d=1$, $H^1(M)\hookrightarrow
    L^\infty(M)$, so the extra nonlinear term $|u|^{2\si}u$ is well
    controlled in view of a limiting procedure, provided that we have
    a uniform bound for $u^\delta$ in $L^\infty(\R_+,H^1(M))$. This is
    the most important step to infer Corollary~\ref{cor:stabnls} from
    the proof of Theorem~\ref{theo:vanish}. 
\smallbreak

Again because $H^1(M)\hookrightarrow    L^\infty(M)$, the global
well-posedness of \eqref{eq:nlsdel2} for $\delta>0$ is
straightforward. The analogues of \eqref{massdel} and
\eqref{energydel} are, since $\l\in \R$:
\begin{align*}
  &\|u^\delta(t)\|_{L^2( M)}^2+2\g\int_0^t\int_{
  M}\frac{|u^\delta(\tau)|^2}{(|u^\delta(\tau)|^2+
\delta)^{\alpha/2}}dxd\tau=\|u_0\|_{L^2( M)}^2, \\
&\frac{d}{dt}\(\|\nabla u^\delta(t)\|_{L^2(M)}^2
+\frac{\l}{\si+1}\|u^\delta(t)\|_{L^{2\si+2}(M)}^{2\si+2}\)=\\
&\quad
-2\g\int_{
    M}\frac{\delta|\nabla
    u^\delta|^2+(1-\alpha)\lvert \RE(\overline{u^\delta}\nabla
    u^\delta)\rvert^2+\lvert\IM(\overline{u^\delta}\nabla
    u^\delta)\rvert^2}{(|u^\delta|^2+\delta)^{\alpha/2+1}}(t,x)dx. 
\end{align*}
They are obtained by the same procedure as in the proof of
Proposition~\ref{cauchydel}: formally, multiply \eqref{eq:nlsdel2} by
$\overline{u^\delta}$, integrate over $M$, and take the imaginary part
to get the first evolution law; multiply \eqref{eq:nlsdel2} by
$\d_t\overline{u^\delta}$, integrate over $M$, and take the real part
to get the second evolution law.
\smallbreak

The first law yields a global \emph{a priori} estimate for
$\|u^\delta(t)\|_{L^2(M)}$, uniformly with respect to $\delta\in
]0,1]$. We infer a uniform $H^1$ control from the second law: in the
same fashion as in the proof of Lemma~\ref{lem:nash},
Gagliardo--Nirenberg inequalities on $\R$ yield, in particular,
\begin{equation*}
  \|f\|_{L^\infty(M)}\le
  C\|f\|_{L^2(M)}^{1/2}\|f\|_{H^1(M)}^{1/2},\quad \forall f\in
  H^1(M). 
\end{equation*}
Denoting 
\begin{equation*}
  E^\delta(t)= \|\nabla u^\delta(t)\|_{L^2(M)}^2
+\frac{\l}{\si+1}\|u^\delta(t)\|_{L^{2\si+2}(M)}^{2\si+2},
\end{equation*}
we have of course $E^\delta(t)\le E^\delta(0)$, a quantity which does
not depend on $\delta$. In the defocusing case $\l\ge 0$, this yields
the required \emph{a priori} $H^1$ estimate. In the focusing case
$\l<0$, write as on $\R$,
\begin{align*}
  \|\nabla u^\delta(t)\|_{L^2(M)}^2 &=E^\delta(t)
  +\frac{|\l|}{\si+1}\|u^\delta(t)\|_{L^{2\si+2}(M)}^{2\si+2} \\
&
\le E^\delta(0) + C
\|u^\delta(t)\|_{L^{\infty}(M)}^{2\si}\|u^\delta(t)\|_{L^{2}(M)}^{2}
\\
&\le C\(\|u_0\|_{H^1(M)}\) +
  C\|u^\delta(t)\|_{L^{2}(M)}^{\si+2}\|u^\delta(t)\|_{H^1(M)}^{\si}. 
\end{align*}
The $L^2$ \emph{a priori} estimate and the assumption $\si<2$ yield a
uniform \emph{a priori} $H^1$ estimate. The analogue of
Theorem~\ref{theo:cauchy} follows as in Section~\ref{sec:exist} (see
\cite{GiVe85} for more details concerning the nonlinearity
$|u|^{2\si}u$).
\smallbreak

Uniqueness stems from the same arguments as in
Section~\ref{sec:unique}, and the fact that the nonlinearity
$|u|^{2\si}u$ is uniformly Lipschitzean on balls of $H^1(M)$, since
\begin{equation*}
  \left\lvert \lvert u\rvert^{2\si}u - \lvert
    v\rvert^{2\si}v\right\rvert \lesssim \(\lvert u\rvert^{2\si} + \lvert
    v\rvert^{2\si}\)\lvert u-v\rvert\lesssim  \(\lVert
    u\rVert_{L^\infty(M)}^{2\si} + \lVert 
    v\rVert_{L^\infty(M)}^{2\si}\)\lvert u-v\rvert,
\end{equation*}
and $H^1(M)\hookrightarrow   L^\infty(M)$.
\smallbreak
The end of the proof of Corollary~\ref{cor:stabnls} is exactly the
same as the proof of Theorem~\ref{theo:vanish} in the case $d=1$:
since $\l\in \R$, \eqref{eq:mass} remains valid.

\bibliographystyle{amsplain}
\bibliography{stab}

\end{document}